\documentclass[11pt,reqno]{amsart}

\usepackage[T1]{fontenc}
\usepackage[english]{babel}

\usepackage{amsmath,amssymb,amscd,amsfonts,mathtools,tikz,caption,float,tikz-cd}
\usetikzlibrary{patterns}
\usepackage[linkcolor=blue,colorlinks=true,citecolor=blue]{hyperref}
\usepackage[capitalise]{cleveref}
\usepackage{colonequals}
\usepackage{wrapfig}

\newtheorem{thm}{Theorem}[section]
\newtheorem{lm}[thm]{Lemma}

\newtheorem{prop}[thm]{Proposition}
\newtheorem{coro}[thm]{Corollary}
\newtheorem{rmk}[thm]{Remark}

\newcommand{\figref}[1]{\hyperref[#1]{Figure \ref{#1}}}
\newcommand{\lemref}[1]{\hyperref[#1]{Lemma \ref{#1}}}
\newcommand{\thmref}[1]{\hyperref[#1]{Theorem \ref{#1}}}
\newcommand{\conjref}[1]{\hyperref[#1]{Conjecture \ref{#1}}}
\newcommand{\propref}[1]{\hyperref[#1]{Proposition \ref{#1}}}
\newcommand{\corref}[1]{\hyperref[#1]{Corollary \ref{#1}}}
\newcommand{\defref}[1]{\hyperref[#1]{Definition \ref{#1}}}
\newcommand{\rmkref}[1]{\hyperref[#1]{Remark \ref{#1}}}
\newcommand{\qref}[1]{\hyperref[#1]{Question \ref{#1}}}
\newcommand{\secref}[1]{\hyperref[#1]{\S\ref{#1}}}
\newcommand{\appref}[1]{\hyperref[#1]{Appendix \ref{#1}}}

\newcommand{\R}{\mathbf{R}}
\newcommand{\C}{\mathbf{C}}

\newcommand{\h}{\mathbf{H}}
\newcommand{\Q}{\mathbf{Q}}
\newcommand{\Z}{\mathbf{Z}}
\newcommand{\N}{\mathbf{N}}

\title{The Brjuno and Wilton Functions}
\author{Claire Burrin}
\author{Seul Bee Lee}
\author{Stefano Marmi}
\thanks{C.B acknowledges the support of SNSF grant number 201557 and the hospitality of SNS Pisa. S.L acknowledges the support of the Institute for Basic Science (IBS-R003-D1) and the support of BK21 SNU Mathematical Sciences Division.
S.M. acknowledges the  support of the research project ``Dynamics and Information Research Institute -
Quantum Information, Quantum Technologies'' within the agreement between UniCredit Bank and Scuola Normale
Superiore.}

\newcommand{\Addresses}{{
  \bigskip
  \footnotesize

  Claire Burrin, \textsc{Institute of Mathematics, University of Zurich, Zurich, Switzerland.}\par\nopagebreak
  \textit{E-mail address}: \texttt{claire.burrin@math.uzh.ch}

  \medskip

  Seul Bee Lee, \textsc{Department of Mathematical Sciences, Seoul National University, 1, Gwanak-ro, Gwanak-gu, Seoul 08826, Korea}\par\nopagebreak
  \textit{E-mail address}: \texttt{seulbee.lee@snu.ac.kr}

  \medskip
  
  Stefano Marmi, \textsc{Classe di Scienze, Scuola Normale Superiore, Pisa, Italy.}\par\nopagebreak
  \textit{E-mail address}: \texttt{stefano.marmi@sns.it}
}
}

\begin{document}
\maketitle

\begin{abstract}
    The Brjuno and Wilton functions bear a striking resemblance, despite their very different origins; while the Brjuno function $B(x)$ is a fundamental tool in one-dimensional holomorphic dynamics, the Wilton function $W(x)$ stems from the study of divisor sums and self-correlation functions in analytic number theory. We show that these perspectives are unified by the semi-Brjuno function $B_0(x)$. Namely, $B(x)$ and $W(x)$ can be expressed in terms of the even and odd parts of $B_0(x)$, respectively, up to a bounded defect. Based on numerical observations, we further analyze the arising functions $\Delta^+(x) = B^+(x) - 2B_0^+(x)$ and $\Delta^-(x) = W^-(x) - 2B_0^-(x)$, the first of which is Hölder continuous whereas the second exhibits discontinuities at rationals, behaving similarly to the classical popcorn function.
\end{abstract}

\section{Introduction}

Complex dynamics, the study of iterates $(f^n(z))_{n\in\N}$ for a holomorphic function $f:\C\to\C$, originated from Schr\"oder's question of whether a holomorphic function $f$ with fixed point $z=0$ is locally conjugate to the linear map $z\mapsto f'(0)\, z$. This was answered positively in the late 19th century by Koenigs \cite{Koenigs1884} when $\lambda:=f'(0)$ has modulus $|\lambda|\in(0,1)$ or $|\lambda|>1$. In the indifferent case $\lambda=e^{2\pi ix}$ where $x$ is irrational,
Siegel \cite{Siegel1942} showed that if $x$ is a Diophantine number, then $f$ is linearizable. Brjuno \cite{Brjuno1965} introduced a weaker condition on $x$ (which we review in the next paragraph) under which $f$ is again locally linearizable. Yoccoz \cite{Yoccoz1995} showed that Brjuno's condition is, in fact, sharp: the quadratic map $f(z)=z^2+\lambda z$ is locally linearizable at $z=0$ if and only if $x$ is a Brjuno number. 
The Brjuno condition is known to be optimal also for other analytic small divisor problems \cite{Davie1994, BG01, Yoccoz2002}.

Following \cite{Yoccoz1995}, the set of Brjuno numbers can be characterized as follows. The Brjuno function $B:\R\setminus\Q\to \R\cup\{\infty\}$ is the 1-periodic function that is defined on the set of irrationals in $(0,1)$ by
\begin{align}\label{eq:seriesB}
    B(x) = \sum_{j=0}^\infty (\{x\} A(\{x\})\cdots A^{j-1}(\{x\})) \log \frac{1}{A^j(\{x\})},
\end{align}
where $\{x\}=x-\lfloor x\rfloor$ denotes the fractional part, $A(x)=\{\frac{1}{x}\}$ is the Gauss map, and we apply the convention that the 0-th summand is $-\log x$. In the language of continued fractions, if $x=[a_0,a_1,\ldots , a_k,\ldots]$ and $\frac{p_j}{q_j}=[a_0,a_1,\ldots , a_j]$, 
the term in parenthesis can be expressed as
$\{x\} A(\{x\})\cdots A^{j-1}(\{x\}) = |q_j x - p_j|,$ that is, in terms of the approximation of $x$ by its $j$-th convergent. Equivalently, the Brjuno function is characterized by the functional equations
\begin{align}\label{eq:fteqB}
    B(x+1)=B(x) \quad \text{and}\quad B(x) = -\log x + x B(1/x) \text{ when } x\in(0,1).
\end{align}
Brjuno numbers are the irrationals $x$ at which the Brjuno function $B(x)$ converges. They can also be characterized by imposing that $\sum_{k=0}^{+\infty}\frac{\log q_{k+1}}{q_k}<+\infty$. 


In his 1987 work Yoccoz actually used a modification of the Brjuno function, $B_{1/2}$, that builds on the nearest integer continued fraction map $A_{1/2}(x)=\Vert \frac{1}{x}\Vert_{\Z}$, where $\Vert x\Vert_{\Z}=\min_{p\in\Z}|x-p|$, in place of the Gauss map $A$. 
Yoccoz showed that $B_{1/2}$ and $B$ coincide up to a bounded defect, i.e., 
\begin{align*}
    |B(x)-B_{1/2}(x)| \leq C
\end{align*}
for a uniform constant $C$, thus both functions are finite if and only if their argument is a Brjuno number. 
Since then,  various Brjuno functions have been studied, typically based on different continued fraction developments, with large families shown to coincide with the standard Brjuno function $B$ up to a bounded defect, e.g., \cite{MMY97,LuzziMarmiNakadaNatsui2010,LM21}.  

One variation on the Brjuno function, in fact, long predates it. In 1933, Wilton \cite{Wilton1933} showed that for any irrational number $x$, the series
\begin{align*}
    \sum_{n=1}^\infty \frac{d(n)}{n} \sin(2\pi n x),
\end{align*}
where $d(n)$ is the number of positive divisors of $n$, converges if and only if the `Wilton function'
\begin{align}\label{eq:seriesW}
    W(x)=  \sum_{j=0}^\infty (-1)^j (\{x\} A(\{x\})\cdots A^{j-1}(\{x\})) \log \frac{1}{A^j(\{x\})}
\end{align}
converges. Wilton numbers, i.e., irrationals $x$ for which $W(x)$ converges, also appear in the literature as the points of differentiability of the multiplicative self-correlation function 
$$
\int_0^1 \{t\}\{xt\} \frac{dt}{t^2},
$$
for the fractional part map $t\mapsto\{t\}$, introduced in the context of Nyman's criterion for the Riemann hypothesis \cite{BDBLS05,BM13}. They can also be characterized by asking that $\sum_{k=0}^{+\infty}(-1)^k \frac{\log q_{k+1}}{q_k}<+\infty$. 

\begin{figure}[h]
\centering
\includegraphics[scale=.25]
{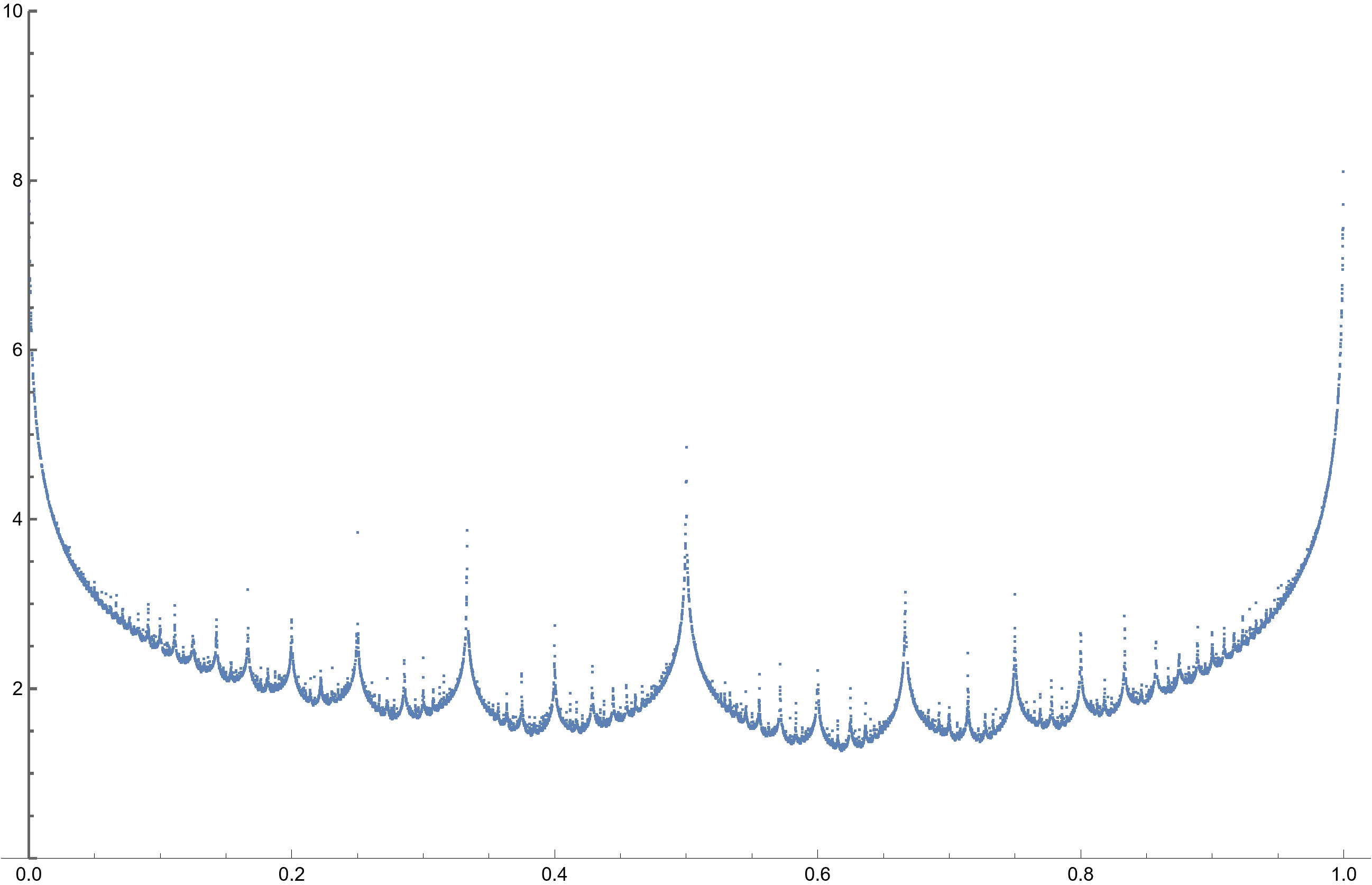}
\includegraphics[scale=.25]
{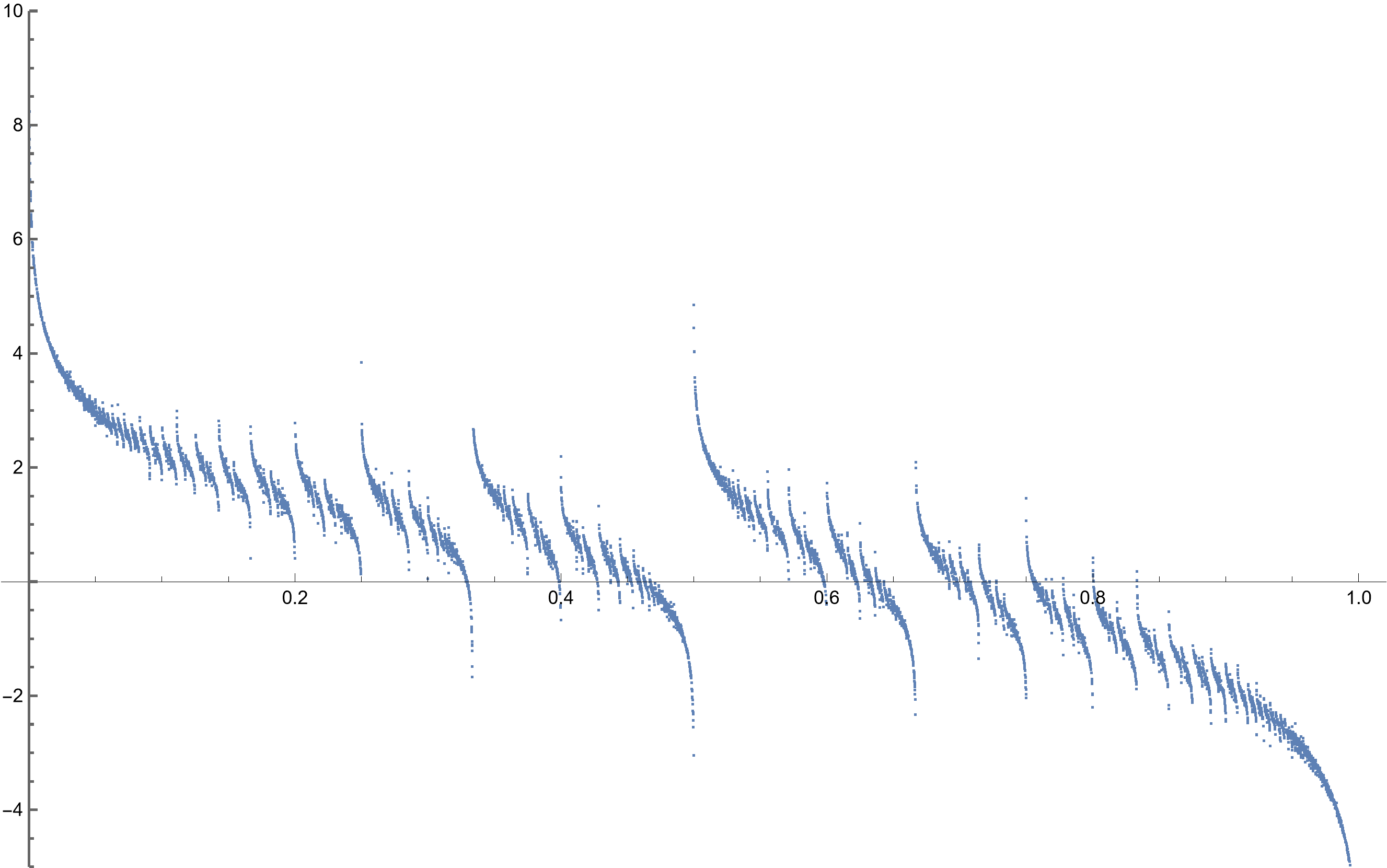}
\caption{Plots of $B$ and $W$ at 10000 randomly chosen points}
\end{figure}

Until now the relation between the Wilton and Brjuno functions has remained mysterious. We show that both functions can be seen as two sides of one coin, and that this coin is the semi-Brjuno function $B_0$ which is the $1$-periodic even function defined by
\begin{align}\label{eq:seriesB0}
    B_{0}(x) =  \sum_{j=0}^\infty (\{x\} A_0(\{x\})\cdots A_0^{j-1}(\{x\})) \log \frac{1}{A_0^j(\{x\})}, ~x\in(0,1),
\end{align}
where $A_0(x)= \lfloor \frac{1}{x}+1\rfloor - \frac{1}{x}$ is the by-excess map. 

\begin{figure}[h]
\centering
\includegraphics[scale=.4]
{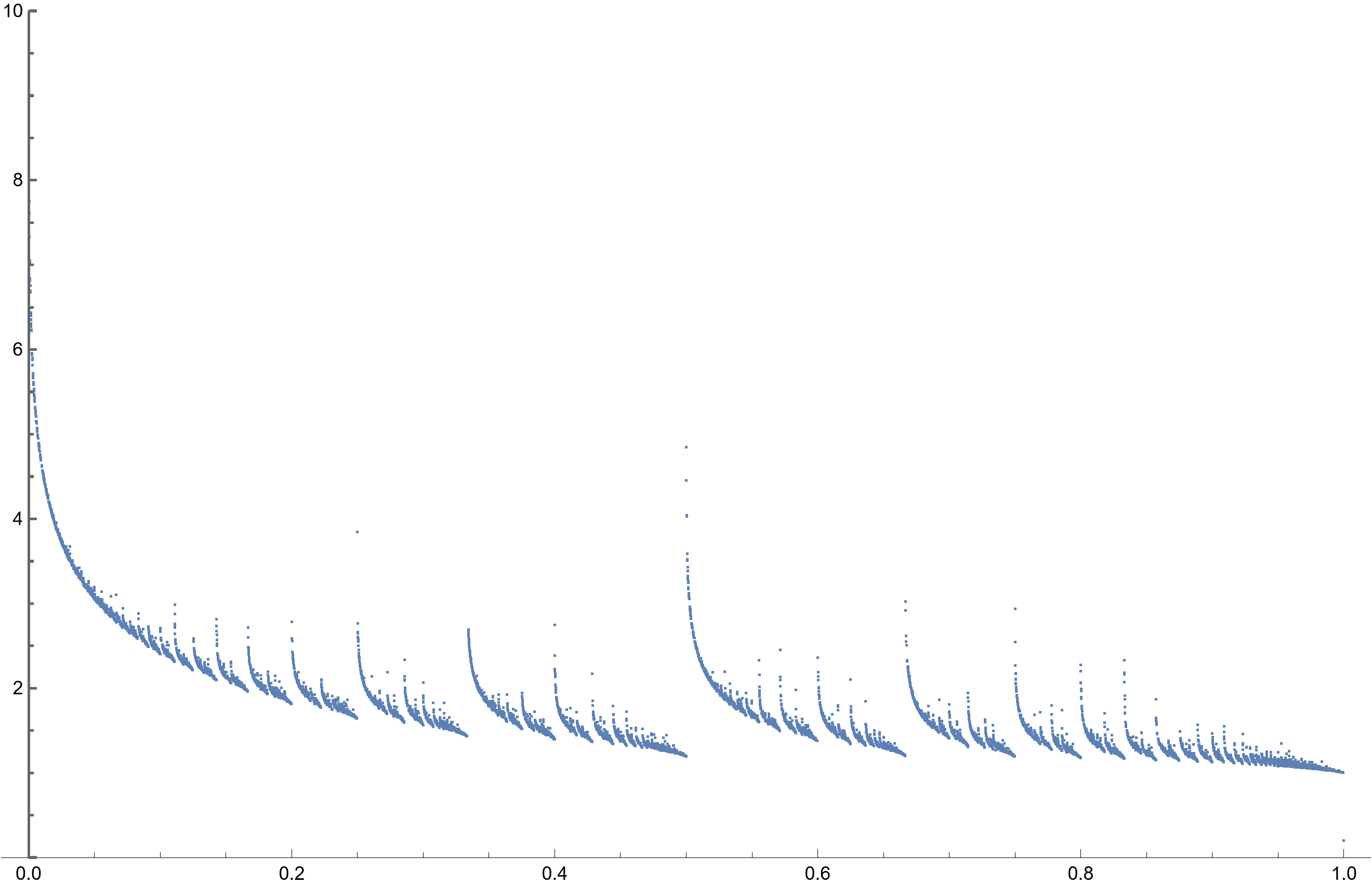}
\caption{Plot of $B_0$ at 10000 randomly chosen points}
\end{figure}

\begin{thm}\label{thm:1}
Let $B_0^\pm$ denote the even and odd parts of the semi-Brjuno function, i.e., $B_0^\pm(x) \coloneqq \frac12 (B_0(x) \pm B_0(-x)).$ Then there exist constants $C^\pm$ such that 
\begin{align*}
    |B(x)-2B_0^+(x)| \leq C^+ \quad \text{and} \quad |W(x)-2B_0^-(x)| \leq C^-
\end{align*} 
uniformly for all irrational numbers $x$.
\end{thm}

The theorem can be deduced from the work of Wilton, by a modification of his application of Voronoi's summation formula, but the proof we present here will be elementary in the sense that it relies on no complex analysis, and in fact only on elementary properties of $\alpha$-continued fractions as established, e.g., in \cite{LuzziMarmiNakadaNatsui2010}. 

Since the odd part $B^-$ and the even part $W^+$ of the Brjuno respectively Wilton functions are uniformly bounded functions --- see Proposition \ref{prop:odd/even} for closed form expressions --- the functions of interest following Theorem \ref{thm:1} are the differences 
$$
\Delta^+(x) := B^+(x)-2B_0^+(x)\quad \text{ and }\quad  \Delta^-(x) := W^-(x)-2B_0^-(x).
$$
Numerical computations show that $\Delta^\pm$ behave widely differently.
\begin{figure}[h]
\centering
\includegraphics[scale=.25]
{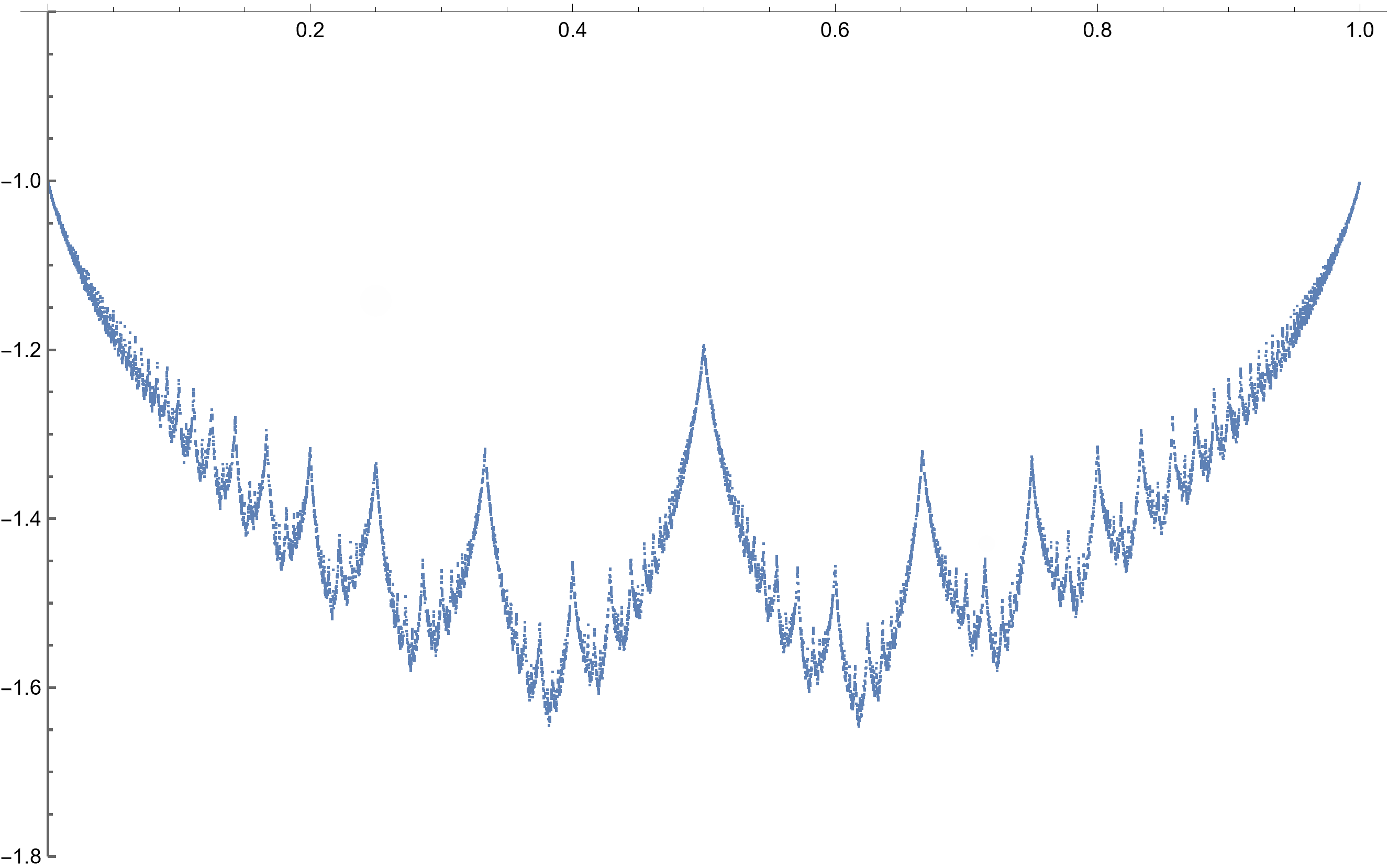}
\includegraphics[scale=.25]
{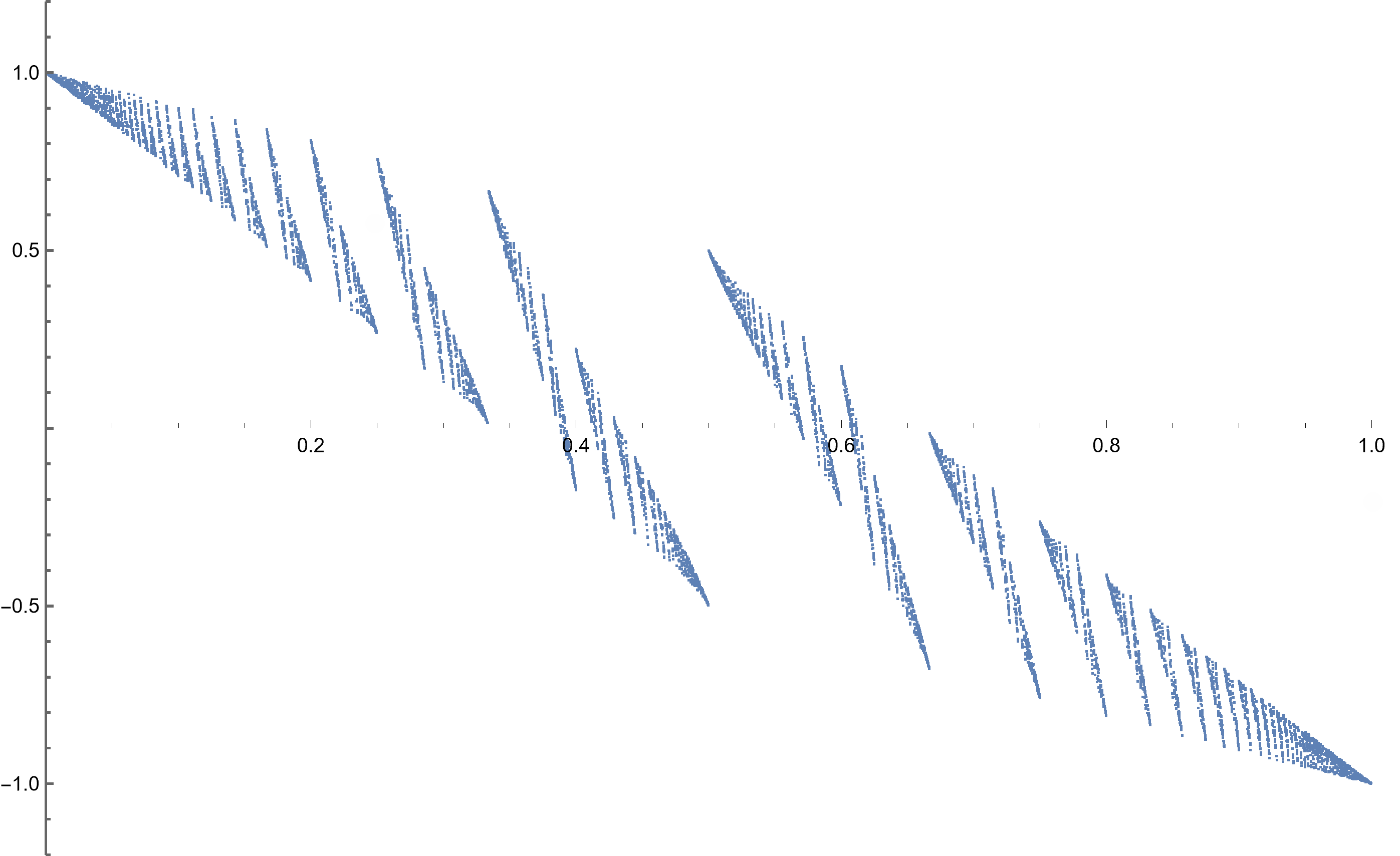}
\caption{Plots of $\Delta^+$ and $\Delta^-$ at 10000 randomly chosen points}
\end{figure}

In \cite[Theorem 1]{LM21}, the second and third authors established that $\Delta^+(x)$ extends to a $\tfrac12$-H\"older continuous function on the real line. This is in stark contrast to what one observes in the graphical plots of $\Delta^-(x)$, which seems to indicate that $\Delta^-(x)$ is discontinuous at each rational with a jump of order $2/q$ for $x=p/q$. In that respect, it is reminiscent of the `popcorn function'
$$
f(x) = \begin{dcases} 1/q & \text{ if } x=p/q\text{ with } (p,q)=1;\\
0 & \text{ if } x \text{ is irrational;}
\end{dcases}
$$
that staple of calculus illustrating the existence of a Riemann integrable function with infinitely many discontinuities. We confirm these numerical observations with the following theorem.

\begin{thm}\label{thm:2}
    The function $\Delta^-(x)$ is continuous on $(0,\tfrac12)\setminus \Q$. At each rational $x=p/q$, it has an increasing jump of $2/q$. In particular $\Delta^-(x)$ is Riemann integrable and nowhere differentiable.
\end{thm}
The proof proceeds by identifying the characterizing functional equation for $\Delta^-$, i.e., determining the function $f(x)=\Delta^-(x)+x\Delta^-(1/x)$ for $x\in(0,\tfrac12)\setminus\Q$. From there, we deduce a series expansion for $\Delta^-$, similarly to (\ref{eq:seriesB})--(\ref{eq:seriesB0}). To show the series thus obtained converges uniformly to a continuous function on $(0,\tfrac12)\setminus\Q$ with the predicted jumps at each rational, we study the regularity of $f$ on the domain of each branch of the iterates of the nearest-integer map $A_{1/2}$.

The third author with Moussa and Yoccoz \cite{MarmiMoussaYoccoz2001} showed that the Brjuno function can be extended to a holomorphic function on $\h/\Z$, the complex Brjuno function, and the same was later established for the Wilton function by the second and third authors with Petrykiewicz and Schindler \cite{LeeMarmiPetrykiewiczSchindler2023}. Following Theorem \ref{thm:1}, we derive at the end of this note a closed form expression for the complex counterpart of $(B+W)/2$.

\section{Even/odd parts and functional equations}\label{sec:even/odd}

Similarly to the Brjuno function, the Wilton function is characterized by functional equations, given by
\begin{align*}
    W(x+1)=W(x)\quad \text{ and }\quad W(x) = -\log x -xW(1/x) \text{ when } x\in(0,1).
\end{align*}
Since $B^-$, $W^+$ are 1-periodic odd respectively even functions, they are completely determined by their values on irrational numbers on the interval $[0,\tfrac12]$. We record the following explicit formulas. 

\begin{prop}\label{prop:odd/even}
Let $x\in(0,\tfrac12)$ be irrational. Then 
\begin{align*}
    B^-(x) = \frac{x}{2} \log \frac{1-x}{x}\quad \text{and} \quad 
    W^+(x) = \frac{x}{2}\log \frac{1-x}{x} - \log(1-x).
\end{align*}
In particular, $B^-$ and $W^+$ are uniformly bounded functions on $\R\setminus\Q$.
\end{prop}


\begin{proof}
We use the functional equations (\ref{eq:fteqB}) for the Brjuno function to rewrite
\begin{align*}
   B(x) - B(-x) = B(x) - B(1-x) &= B(x) +\log(1-x) - (1-x) B\left(\frac{1}{1-x}\right).
\end{align*}
Since $x\in (0, \frac{1}{2})$ we have $\frac{1}{1-x}-1=\frac{x}{1-x}\in (0,1)$ and so we may apply the functional equation again
\begin{align*}
   &= B(x) + \log(1-x) - (1-x)\log \frac{1-x}{x} - x B\left(\frac{1}{x}\right) = x \log \frac{1-x}{x}.
\end{align*}
The proof is the same for the even part of the Wilton function, based on its functional equations.
\end{proof}

\begin{prop}[Functional equation for $W^-$]\label{prop:fteq for W-}
Let $x\in(0,\frac12)$ be irrational and define
\begin{align*}
g(x)&:= -\log x - [ W^+(x)+ xW^+(Ax)].
\end{align*}
Then 
$
W^-(x)= g(x) -xW^-\left(Ax\right) 
$ 
for every  $x\in(0,\frac12)\setminus\Q$.
\end{prop}
\begin{proof} 
From the functional equation of the Wilton function and its decomposition in even and odd parts, we get that
$$
W^-(x) = W(x)-W^+(x) = -\log(x) -x W^+(1/x) - W^+(x) -xW^-(1/x).
$$
\end{proof}

\begin{rmk}
    The function $g(x)$ can be made explicit by inserting the expression for $W^+(x)$ given by Proposition \ref{prop:odd/even}.
\end{rmk}

\begin{coro}[Functional equation for $\Delta^-$]\label{coro:funct-eq}
Recall the definition of $g(x)$ in Proposition \ref{prop:fteq for W-} and set
$$
\Phi(x):= xB_0(Ax)-B_0(1-x), \quad f(x):= g(x)+ \log x -\Phi(x)
$$
for each $x\in(0,\tfrac12)\setminus\Q$. The function $\Delta^-(x)=W^-(x)-2B_0^-(x)$ is odd and 1-periodic on $\R\setminus\Q$ and satisfies the functional equation $$
\Delta^-(x)=f(x)-x\Delta^-(A x)
$$
for $x\in(0,\tfrac12)\setminus\Q$.
\end{coro}
\begin{proof}
By definition we have $\Delta^-(x) = W^-(x)-B_0(x)+B_0(1-x)$, hence
\begin{align*}
    \Delta^-(x)+x\Delta^-(Ax) &= W^-(x)+xW^-(Ax)-B_0(x)+xB_0(1-A x)\\
    &\quad - xB_0(Ax) + B_0(1-x).
\end{align*}
For all $x\in(0,1)$,
$A_0(x)=1-A(x)$ and so by Proposition \ref{prop:fteq for W-} and the functional equation for $B_0$, this is $=g(x)+\log x -\Phi(x)$.
\end{proof}
\begin{rmk}
For $x\in(0,\tfrac12)\setminus\Q$ we have the explicit form
\begin{align}\label{eq:closedformPhi}
    \Phi(x) = x\log x + x A(x) \log(xA(x)) + x \sum_{j=1}^{n-1} \log(1-jx),
\end{align}
with $n=\lfloor \tfrac{1}{x}\rfloor$. This is proven in part {\rm (1)} of the proof of Theorem 1 in \cite{LM21}.
\end{rmk}

\section{Proof of Theorem \ref{thm:1}}

Brjuno, Wilton and semi-Brjuno numbers can be tested using the following valid inequalities
\begin{align*}
\left| B(x) - \sum_{j\geq0} \frac{\log q_{j+1}}{q_j} \right|,\, &\left| W(x) -\sum_{j\geq0} (-1)^j\frac{\log q_{j+1}}{q_j} \right|,\   \left| B_0(x) -\sum_{j\geq0} \frac{\log a_{2j+1}}{q_{2j}} \right|  \leq C
\end{align*}
for some uniform constant $C$, where $a_j$ and $q_j$ are the partial quotients, respectively denominators of the convergents for the simple continued fraction expansion of $x$. For the Brjuno function, this is due to Yoccoz \cite[p.~14]{Yoccoz1995}, with the same argument giving the statement for the Wilton function. For the semi-Brjuno function this is established by combining Theorems 13 and 15 in \cite{LuzziMarmiNakadaNatsui2010}. Hence 
\begin{align*}
    \frac{B(x)+W(x)}{2}\, &\approx\, \sum_{j\geq0} \frac{\log q_{2j+1}}{q_{2j}}\, =\, \sum_{j\geq0} \frac{\log (a_{2j+1}(q_{2j}+\tfrac{q_{2j-1}}{a_{2j+1}}))}{q_{2j}}\\
    &\approx\,  \sum_{j\geq0} \frac{\log a_{2j+1}}{q_{2j}}\, \approx\, B_0(x),
\end{align*}
where we use $\approx$ to denote that left and right hand-sides differ by a uniform constant. 
Here we use the fact that $\sum_{j\ge0} \frac{\log q_{j}}{q_j}$ is uniformly bounded, see Remark 1.7 in \cite{MMY97}.
Together with \cref{prop:odd/even}, this yields
\begin{align*}
    2B_0^+(x) \approx B^+(x)+W^+(x) \approx B^+(x) \approx B^+(x) + B^-(x) = B(x)
\end{align*}
and $2B_0^-(x)\approx W(x)$. This concludes the proof of Theorem \ref{thm:1}.

\vspace{.1cm}

\section{Proof of Theorem \ref{thm:2}}

Recall the definitions of $f(x)$, $g(x)$ from Proposition \ref{prop:fteq for W-}, Corollary \ref{coro:funct-eq}. We will denote by $\tilde f(x)$ the extension of $f(x)$ to a 1-periodic odd function on $\R\setminus\Q$.

\begin{lm}\label{lm:continuity}
The function $\tilde f(x)$ is continuous for every  $x\in\R\setminus\Q$, with 
$$
\lim_{x\to 0^+} \tilde f(x) = 1,\qquad \lim_{x\to \frac12^-} \tilde f(x) = 0,
$$
where the limit is taken on irrational numbers.
\end{lm}

\begin{proof}
In the proof of Theorem 1 in \cite{LM21}, the function $\Phi(x) = g(x) + \log x -f(x)$ (see Corollary \ref{coro:funct-eq} and Figure~\ref{fig:rest.B})
 is shown to extend to an $\alpha$-H\"older continuous function on $[0,\frac12]$ for $\tfrac12<\alpha<1$. 
 \begin{figure}[ht]
\includegraphics[width=8cm]{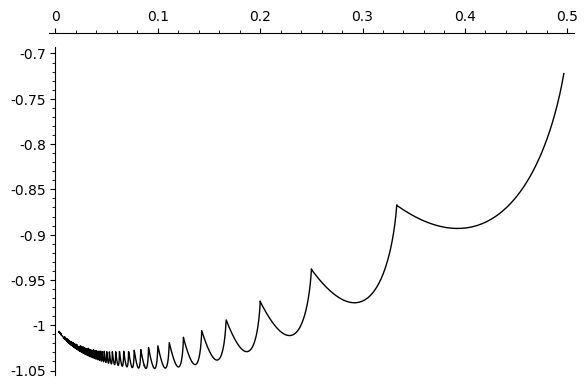}
\caption{The graph of $\Phi$ on $(0,\frac12)$.}
\label{fig:rest.B}
\end{figure}
 Based on the closed-form expression (\ref{eq:closedformPhi}) we have 
 $$
 \lim_{x\to0^+} \Phi(x) = -1\,  \quad \text{ and }\quad \lim_{x\to\frac12^-} \Phi(x) = -\log 2.
 $$
 On the other hand, the definition of $g(x)$ with Proposition \ref{prop:odd/even} together with the fact that $W^+$ is uniformly bounded on irrationals for this need limit over irrationals implies that 
 $$
 \lim_{x\to0^+} (g(x)+\log x) = 0\,  \quad \text{ and }\quad \lim_{x\to\frac12^-} (g(x)+\log x) = -\log 2.
 $$
\end{proof}

The nearest-integer map
\begin{equation}\label{eq:NICF}
A_{1/2}:(0,1/2)\to(0,1/2), \quad A_{1/2}(x) = \left|\frac{1}{x}-\left\lfloor \frac{1}{x}+\frac{1}{2}\right\rfloor\right|.
\end{equation}
yields the continued fraction expansion 
$$x = a_0+\frac{\varepsilon_0}{a_1+\frac{\varepsilon_1}{\ddots}},$$
where 
$$
a_i = \left\lfloor\frac{1}{x_{i-1}}+\frac{1}{2}\right\rfloor,\, \varepsilon_i = \mathrm{sgn}\left(\frac{1}{x_{i-1}}-a_i\right),\,  x_i = A_{1/2}(x_{i-1}) = \varepsilon_i\left(\frac{1}{x_{i-1}}-a_i\right)
$$
for $i\geq1$ with initial seed $x_{0}=|x-\lfloor x +1/2\rfloor|$, $a_0= \lfloor x +1/2\rfloor$ and $\varepsilon_0=\mathrm{sgn}(x-a_0)$.
For $x\in(0,1/2)$, we have 
$$\begin{pmatrix}\varepsilon_nP_{n-1}&P_n\\ \varepsilon_nQ_{n-1}&Q_n\end{pmatrix} = \begin{pmatrix}0&1\\ \varepsilon_1&a_1\end{pmatrix}\cdots \begin{pmatrix}0&1\\ \varepsilon_n&a_n\end{pmatrix}.$$
Thus
$$x_n = (-\varepsilon_n)\frac{Q_nx-P_n}{Q_{n-1}x-P_{n-1}}.$$
Then we have $x_0x_1\cdots x_n = |Q_nx-P_n|$ and $|Q_nP_{n-1}-Q_{n-1}P_n|=1$.




\begin{thm} \label{jumps}
The function $\Delta^-$ is the series
\begin{align}\label{eq:series}
\Delta^-(x) = \sum_{i=0}^\infty (-1)^i \beta_{i-1}(x)\tilde{f}(\varepsilon_ix_i),
\end{align}
where $\beta_i(x) = \varepsilon_0x_0\varepsilon_1x_1\cdots \varepsilon_{i}x_{i}$ and $\beta_{-1}(x)=1$, and it converges uniformly. Moreover, for each $M\ge 1$, its truncation 
$$\Delta^-_M(x) = \sum_{i=0}^M(-1)^i\varepsilon_0x_0\varepsilon_1x_1\cdots \varepsilon_{i-1}x_{i-1}\tilde{f}(\varepsilon_ix_i)$$
is continuous on $\mathbb (0,\frac 12)\setminus\Q_M$, where $\Q_M$ is the set of rationals $p/q$ whose depth ($i$ such that $A_{1/2}^i(p/q)=0$) is at most $M$.
On $p/q\in\Q_M$, $\Delta^-_M$ has increasing jump of $2/q$.
\end{thm}

\begin{proof}
 Since $\Delta^-(x)$ is 1-periodic and odd on $\R\setminus\Q$ (Corollary \ref{coro:funct-eq}), we have 
\begin{align*}
    \Delta^-(x) &= \varepsilon_0 \Delta^-(\varepsilon_0(x-a_0))=\varepsilon_0 \Delta^-(x_0) = \varepsilon_0\left[f(x_0)-x_0 \Delta^-\left(\frac{1}{x_0}\right)\right]\\ 
    &= \varepsilon_0 f(x_0)-\varepsilon_0x_0 \Delta^-\left(\frac{1}{x_0}-a_1\right) = \tilde{f}(x) - \varepsilon_0x_0 \Delta^-(\varepsilon_1x_1)
\end{align*}
and this yields the expansion (\ref{eq:series}). By \cite[Proposition 1.4]{MMY97} we have $|\beta_{N}(x)|\leq \frac12 \gamma^N$ uniformly, with $\gamma=\sqrt{2}-1$, and since $\tilde{f}$ is uniformly bounded, there exists $C>0$ such that $$|\sum_{i\ge N}(-1)^i\beta_{i-1}(x)\tilde{f}(\varepsilon_ix_i)|\le C\frac{\gamma^N}{1-\gamma}\to 0$$ as $N\to \infty$. Hence the series converges uniformly.

Let $x\in(0,\frac12)$. Then $\varepsilon_0=1$ and $x_0=x$. We consider the set of open intervals $\mathcal P_i$, $i\ge1$, determined by the domain of each branch of $A_{1/2}^i$.

Let $i=1$. Then
$\mathcal P_1 =\{(\frac{1}{n+1},\frac{2}{2n+1}),(\frac{2}{2n+1},\frac{1}{n})\}_{n\ge 2}$, $\Q_1 = \{\frac{1}{n}\}_{n\ge 2}$. The function $\Delta^-_0 = \tilde{f}(\varepsilon_0x_0) = f(x)$ is continuous on $(0,\frac 12)$.
%
The $i$th term of $\Delta^-$ is
$$(-1)^i\varepsilon_0x_0\varepsilon_1x_1\cdots \varepsilon_{i-1}x_{i-1}\tilde{f}(\varepsilon_ix_i),$$
which is continuous on each open interval of $\mathcal P_i$.
One endpoint of an interval of $\mathcal P_i$ has depth $i$ and the other has depth $i+1$.
Let $p/q\in \Q_i\setminus \Q_{i-1}$.
Let $I, I'$ be the intervals of $\mathcal{P}_i$ on the left respectively right sides of $p/q$.
Then the nearest integer continued fraction of $x\in I\cup I'$ has the same $a_j$ for $j=1,\cdots,i$ and $\varepsilon_j$ for $j=1,\cdots,i-1$ with $p/q$.

Since $x_0x_1\cdots x_{i-1}=|Q_{i-1}x-P_{i-1}|$, it is continuous at $p/q$ and its value is $1/q$. 

The $i$th iteration $x_i(x) = A_{1/2}^i(x)$ is continuous and $\lim_{x\to p/q} x_i(x)= 0^+$.
Thus $\lim_{x\to p/q} \tilde{f}(x_i(x))=1$.

Also, we have
$$(-1)^i\varepsilon_1\cdots\varepsilon_i =\begin{cases}-1&\text{if }x\in I,\\
1&\text{if }x\in I'.
\end{cases}
$$
Thus $\lim_{x\to p/q^+} (-1)^i\beta_{i-1}(x)\tilde{f}(x_i(x))=-1/q$ and $\lim_{x\to p/q^-} (-1)^i\beta_{i-1}(x)\tilde{f}(x_i(x))=1/q$.
\end{proof}

\begin{coro}
The series $\Delta^-(x) = \sum (-1)^i \beta_{i-1}(x)\tilde{f}(\varepsilon_ix_i)$ converges to a continuous function on $ (0,\frac 12)\setminus\Q$.
At each rational number $x=\frac{p}{q}$, the function $\Delta^-$ has right and left limits and an increasing jump of $\frac{2}{q}$. 
\end{coro}

\begin{proof}
The first statement is a consequence of of the continuity of the partial sums $\Delta^-_M$ on $\mathbb (0,\frac 12)\setminus\Q$ and their uniform convergence to $\Delta^-$.
In fact, if $x_0\in \mathbb (0,\frac 12)\setminus\Q$ and $\varepsilon >0$, by uniform convergence,
there exists $n_0(\varepsilon ) >0$ such that for all $n\ge n_0$ one has 
$|\Delta^-_n(x)-\Delta^- (x)|\le \frac{\varepsilon}{3}$ for all $x\in \mathbb (0,\frac 12)\setminus\Q$. On the other hand, for all $n>n_0$ fixed,  there exists $\delta_n>0$ such that for all $x\in (x_0-\delta, x_0+\delta)\setminus\Q$ one has $|\Delta^-_n(x)-\Delta^-_n(x_0)|\le \frac{\varepsilon}{3}$. Combining the two estimates one obtains 
$ 
|\Delta^- (x)-\Delta^- (x_0)|\le \varepsilon 
$
thus $\lim_{x\rightarrow x_0}\Delta^- (x)=\Delta^- (x_0)$ where the limit is taken on irrational numbers. The statement on the jumps follows from the uniform convergence of $\Delta_M^-$.
\end{proof}

\section{Complex Brjuno and Wilton functions}

Following Theorem \ref{thm:1} we may consider (up to a finite defect) the semi-Brjuno function $B_0$ as the linear combination $(B+W)/2$ of the standard Brjuno and Wilton functions. Both have a complex counterpart, $\mathcal{B},\mathcal{W}:\mathbb{H}\to\C$, such that if $\alpha$ is a Brjuno number, then $\mathrm{Im} \mathcal{B}(\alpha+w)$ converges to $B(\alpha)$ as $w\to 0$ in any domain with a finite order of tangency to the real axis, and similarly for the Wilton functions if $\alpha$ is a Wilton number \cite{MarmiMoussaYoccoz2001,LeeMarmiPetrykiewiczSchindler2023}. In particular, we have that $\tfrac12\mathrm{Im}(\mathcal{B}+\mathcal{W})(\alpha+w)$ converges to $B_0(\alpha)$ up to a finite defect, when $\alpha$ is both Brjuno and Wilton, as $w\to0$ in any domain with a finite order of tangency to the real axis. Moreover, the complex Brjuno and Wilton have explicit closed form expressions, given by \cite[(1.7)]{MarmiMoussaYoccoz2001} \cite[Corrected (1.11)]{LeeMarmiPetrykiewiczSchindler2023}
\begin{align*}
\mathcal{B}(z) &= -\frac{1}{\pi} \sum_{p/q\in \Q}  (p'-q' z)\left[ \mathrm{Li}_2\left(\frac{p'-q'z}{qz-p}\right)-\mathrm{Li}_2\left(-\frac{q'}{q}\right)\right]\\
&\quad +(p''-q'' z)\left[\mathrm{Li}_2\left(\frac{p''-q'' z}{qz-p}\right)-\mathrm{Li}_2\left(-\frac{q''}{q}\right)\right] +\frac{1}{q}\log \frac{q+q''}{q+q'}
\end{align*}
and
\begin{align*}
\mathcal{W}(z) &= -\frac{1}{\pi} \sum_{p/q\in \Q}  (q' z-p')\left[ \mathrm{Li}_2\left(\frac{p'-q'z}{qz-p}\right)-\mathrm{Li}_2\left(-\frac{q'}{q}\right)\right]\\
&\quad +(p''-q'' z)\left[\mathrm{Li}_2\left(\frac{p''-q'' z}{qz-p}\right)-\mathrm{Li}_2\left(-\frac{q''}{q}\right)\right] +\frac{1}{q}\log \frac{(q+q')(q+q'')}{q^2},
\end{align*}
where $\mathrm{Li}_2$ is the dilogarithm
\cite{Zagier2007} and 
$[\tfrac{p'}{q'},\tfrac{p''}{q''}]$ is the Farey interval such that $\tfrac{p}{q}=\tfrac{p'+p''}{q'+q''}$ (with the convention that $p'=p-1$, $q'=1$, $p''=1$, $q''=0$ if $q=1$). It follows that
\begin{align*}
    \frac12 (\mathcal{B}+\mathcal{W}) &= -\frac{1}{\pi} \sum_{p/q\in \Q}  (p''-q'' z)\left[ \mathrm{Li}_2\left(\frac{p''-q''z}{qz-p}\right)-\mathrm{Li}_2\left(-\frac{q''}{q}\right)\right]   +\frac{1}{q}\log \frac{q+q''}{q}.
\end{align*}

\Addresses


\begin{thebibliography}{99}


\bibitem[BDBLS05]{BDBLS05}
L. Báez-Duarte, M. Balazard, B. Landreau, and E. Saias, Étude de l’autocorrélation multiplicative de la fonction ‘partie fractionnaire’. The Ramanujan Journal (2005), 9(1), 215-240.

\bibitem[BM13]{BM13}
M. Balazard and B. Martin, Sur l'autocorrélation multiplicative de la fonction "partie fractionnaire" et une fonction définie par J. R. Wilton, arXiv:1305.4395v1.

\bibitem[BG01]{BG01}
A. Berretti and G. Gentile, 
Bryuno function and the standard map. 
Communications in Mathematical Physics (2001), 220, 623--656

\bibitem[Brj65]{Brjuno1965}
A. D. Brjuno, On convergence of transforms of differential equations to the normal form (Russian), Dokl. Akad. Nauk SSSR (Russian Academy of Sciences) 165 (1965), 987–989.

\bibitem[Da94]{Davie1994}
A.M. Davie, The critical function for the semistandard map, Nonlinearity (1994), 7(1), 219--229

\bibitem[Koe84]{Koenigs1884}
G. Koenigs, Recherches sur les intégrales de certaines équations fonctionnelles, Ann. Sci. \'ENS, 3\`eme s\'erie, tome 1 (1884), 3-41.


\bibitem[KN01]{KN01}
C. Kraaikamp and H. Nakada, 
On a problem of Schweiger concerning normal numbers, 
J. Number Theory, 86(2) (2001) 330–340

\bibitem[LMPS23]{LeeMarmiPetrykiewiczSchindler2023} S.B. Lee, S. Marmi, I. Petrykiewicz, T.I. Schindler, Regularity properties of $k$-Brjuno and Wilton functions, Aequationes Mathematicae, 98 (2024) 13-85 and  Correction to: Regularity properties of k-Brjuno and Wilton functions, Aequationes Mathematicae, 98 (2024) 349-350

\bibitem[LM21]{LM21}
S.B. Lee and S. Marmi,
The Brjuno functions of the by-excess, odd, even and odd-odd continued fractions and their regularity properties, arXiv:2111.13553.

\bibitem[LMNN10]{LuzziMarmiNakadaNatsui2010} L. Luzzi, S. Marmi, H. Nakada, R. Natsui, Generalized Brjuno functions associated to $\alpha$-continued fractions, Journal of Approximation Theory, 162(1) (2010) 24-41

\bibitem[MMY97]{MMY97}{
{S. Marmi, P. Moussa and J-C. Yoccoz},
The Brjuno functions and their regularity properties,
{Comm. Math. Phys.},
{\bf 186}(2) (1997) {265--293}.
}

\bibitem[MMY01]{MarmiMoussaYoccoz2001} S. Marmi, P. Moussa, J.-C. Yoccoz, Complex Brjuno functions. Journal of the American Mathematical Society 14(4) (2001) 783-841

\bibitem[Sie42]{Siegel1942}
C.L. Siegel, Iteration of analytic functions. Annals of Mathematics 43(4) (1942) 607-612

\bibitem[Wil33]{Wilton1933} J.R. Wilton, An approximate functional equation with applications to a problem of Diophantine approximation. Journal für die reine und angewandte Mathematik 169 (1933) 219-237

\bibitem[Yoc95]{Yoccoz1995}
J.C. Yoccoz, Th\'eor\`eme de Siegel, nombres de Bruno et polyn\^omes quadratiques, Ast\'erisque, 231 (1995), 3–88.

\bibitem[Yoc02]{Yoccoz2002}
J.C. Yoccoz, Analytic linearization of circle diffeomorphisms,
in Dynamical Systems and Small Divisors: Lectures given at the CIME Summer School held in Cetraro, Italy, June 13-20, 1998,
Lecture Notes in Mathematics, 1784 (2002),
125--173.

\bibitem[Za07]{Zagier2007}
D. Zagier, The dilogarithm function, in 
Frontiers in Number Theory, Physics, and Geometry II: On Conformal Field Theories, Discrete Groups and Renormalization, (2007), 
3--65.
    
\end{thebibliography}
\end{document}